\documentclass{article}

\usepackage{paralist}
\usepackage{amsfonts}
\usepackage{bm}
\usepackage[ruled,vlined]{algorithm2e}
\usepackage{subcaption}
\SetKwProg{myFn}{Function}{}{}
\newcommand{\Fn}[1]{\myFn{\textup{#1}}}
\newlength\mylen
\newcommand\otherinput[1]{%
	\settowidth\mylen{\KwIn{}}%
	\setlength\hangindent{\mylen}%
	\hspace*{\mylen}#1\\}
\newcommand\otheroutput[1]{%
	\settowidth\mylen{\KwOut{}}%
	\setlength\hangindent{\mylen}%
	\hspace*{\mylen}#1\\}
\newcommand{\var}[1]{\textit{#1}}
\SetArgSty{textup}
\usepackage{hyperref}
\usepackage{graphicx}
\usepackage{xcolor, colortbl}

\newcommand\poly[1][]{\mathcal{P}_{#1}}
\newcommand\cons[1][]{\mathcal{C}_{#1}}
\newcommand\region[1][]{\mathcal{R}_{#1}}
\newcommand\frontier[1][]{\mathcal{F}_{#1}}

\newcommand{\algo}[1]{Algorithm~#1}
\newcommand{\glpk}{GLPK\xspace}
\newcommand{\flint}{\textsc{flint}\xspace}
\newcommand{\fig}[1]{Figure~#1}

\newcommand{\seci}[1]{Section~#1}

\newcommand{\equa}[1]{Equation~#1}

\newcommand{\prob}[1]{Problem~#1}

\newcommand{\MQ}{\mathbb{Q}}
\newcommand{\MF}{\mathbb{F}}
\newcommand{\rational}[1][]{\ensuremath{
		\ifthenelse{\equal{#1}{}}{\MQ}{{#1}^{\MQ}}}}
\newcommand{\float}[1][]{\ensuremath{
		\ifthenelse{\equal{#1}{}}{\MF}{{#1}^{\MF}}}}
\newcommand{\bothfields}[1][]{\ensuremath{
		\ifthenelse{\equal{#1}{}}{\MQ \times \MF}{{#1}^{\MQ \times \MF}}}}
	
\newcommand{\none}{ \textsf{none}}
\usepackage{mathtools}

\usepackage{authblk}

\title{An efficient parametric linear programming solver and application to polyhedral projection}

\author{Hang Yu}
\author{David Monniaux}
\affil{
	Univ. Grenoble Alpes, CNRS, Grenoble INP\thanks{Institute of Engineering Univ. Grenoble Alpes}\\
	F-38000 Grenoble, France\\
	\url{First-name.Last-name@univ-grenoble-alpes.fr} \\
}
\date{October 8, 2019}

\usepackage[backend=bibtex]{biblatex}
\usepackage{dmbiblatex}
\bibliography{sas_2019}

\usepackage{amsthm}
\newtheorem{theorem}{Theorem}
\theoremstyle{definition}
\newtheorem{definition}{Definition}

\usepackage{hyperref}

\begin{document}
\maketitle              
\begin{abstract}
  Polyhedral projection is a main operation of the polyhedron abstract domain.
It can be computed via parametric linear programming (PLP), which is more efficient than the classic Fourier-Motzkin elimination method.

In prior work, PLP was done in arbitrary precision rational arithmetic.
In this paper, we present an approach where most of the computation is performed in floating-point arithmetic, then exact rational results are reconstructed.

We also propose a workaround for a difficulty that plagued previous attempts at using PLP for computations on polyhedra: in general the linear programming problems are degenerate, resulting in redundant computations and geometric descriptions.
\end{abstract}

\section{Introduction and related work}
Abstract interpretation \cite{cousot_1977_popl} is an approach for obtaining invariant properties of programs, which may be used to verify their correctness.
Abstract interpretation searches for invariants within an \textit{abstract domain}.
For numerical properties, a common and cheap choice is one interval per variable per location in the program, but this cannot represent relationships between variables.
Such imprecision often makes it impossible to prove properties of the program using that domain. 
If we retain linear equalities and inequalities between variables, we obtain the domain of \emph{convex polyhedra} \cite{cousot_1978_automatic}, which is more expensive, but more precise.

Several implementations of the domain of convex polyhedra over the field of rational numbers are available.
The most popular ones for abstract interpretation are NewPolka\footnote{Now distributed as part of APRON \url{http://apron.cri.ensmp.fr/library/}}
and the Parma Polyhedra Library (PPL) \cite{bagnara_2008_ppl}.
These libraries, and others, use the \emph{double description} of polyhedra: as \emph{generators} (vertices, and for unbounded polyhedra, rays and lines) and constraints (linear equalities and inequalities).
Some operations are easier on one representation than on the other, and some, such as removing redundant constraints or generators, are easier if both are available.
One representation is computed from the other using Chernikova's algorithm \cite{chernikova_1968_algorithm,LeVerge94}.
This algorithm is expensive in some cases, and, furthermore, in some cases, one representation is exponentially larger than the other.
This is in particular the case of the generator representation of hypercubes or, more generally, products of intervals;
thus interval analysis which simulate using convex polyhedra in the double description has cost exponential in the dimension.

In 2012 Verimag started implementing a library using constraints only, called VPL (Verified Polyhedra Library)
\cite{fouilhe_2015_phd,marechal_2017_phd}.
There are several reasons for using only constraints; we have already cited the high generator complexity of some polyhedra commonly found in abstract interpretation, and the high cost of Chernikova's algorithm.
Another reason was to be able to certify the results of the computation, in particular that the obtained polyhedra includes the one that should have been computed, which is the property that ensures the soundness of abstract interpretation.
One can certify that each constraint is correct by exhibiting coefficients, as in Farkas' lemma.

In the first version of VPL, all main operations boiled down to projection, performed using Fourier-Motzkin elimination \cite{dantzig_1972_fourier}, but this method generates many redundant constraints which must be eliminated at high cost.
Also, for projecting out many variables $x_1,\dots,x_n$, it computes all intermediate steps (projection of $x_1$, then of $x_2$\dots), even though they may be unneeded and have high description complexity.
In the second version, projection and convex hull both boil down to \emph{parametric linear programming}  \cite{jones_2008_polyhedral}.
The current version of VPL is based on a parametric linear programming solver implemented in arbitrary precision arithmetic in OCaml \cite{marechal_2017_sas}.

In this paper, we improved on this approach in two respects.
\begin{itemize}
\item We replace most of the exact computations in arbitrary precision rational numbers by floating-point computations performed using an off-the-shelf linear programming solver.
  We can however recover exact solutions and check them exactly,
  an approach that has previously been used for SMT-solving \cite{monniaux_2009_using,king_2014_leveraging}.

\item We resolve some difficulties due to geometric degeneracy in the problems to be solved, which previously resulted in many redundant computations.
\end{itemize}
Furthermore, the solving is divided into independent tasks, which may be scheduled in parallel. The parallel implementation is covered in \cite{coti_monniaux_yu_2019}.

\section{Notations and preliminaries}
\subsection{Notations}
Capital letters (e.g. $A$) denote matrices, small bold letters (e.g.~$\bm{x}$) denote vectors, small letters (e.g.~$b$) denote scalars.
The $i$th row of $A$ is $\bm{a}_{i\bullet}$, its $j$th column is $\bm{a}_{\bullet j}$.
$\poly:A\bm{x} + \bm{b} \geq 0$ denotes a polyhedron and $\cons$ a constraint. The $i$th constraint of $\poly$ is $\cons[i]$: $\bm{a}_{i \bullet} \bm{x} \geq b_i$, where $b_i$ is the $i$th element of $\bm{b}$.
$a_{ij}$ denotes the element at the $i$th row and the $j$th column of $A$.
$\rational$ denotes the field of rational numbers, and $\float$ is the set of finite floating-point numbers, considered as a subset of $\rational$.

\subsection{Linear programming}
Linear programming (LP) consists in getting the optimal value of a linear function $Z(\bm{\lambda})$ subject to a set of linear constraints $A\bm{\lambda} = \bm{b},\ \bm{\lambda} \geq 0$~\footnote{This is the canonical form of the LP problem. All the LP problems can be transformed into this form.},
where $\bm{\lambda}$ is the vector of variables.
The optimal value $Z^*$ is reached at $\bm{\lambda^*}$:
$Z^* = Z(\bm{\lambda^*})$.

\subsection{Basic and non-basic variables}
We use the implementation of the simplex algorithm in \glpk%
\footnote{The GNU Linear Programming Toolkit (GLPK) is a linear programming solver implemented in floating-point arithmetic. \url{https://www.gnu.org/software/glpk/}}
as LP solver.
In the simplex algorithm each constraint is expressed in the form
$(\lambda_B)_i = \sum_{j=1}^{n} a_{ij}(\lambda_N)_j+c_i$, where $(\lambda_B)_i$ is known as a \emph{basic variable}, the $(\lambda_N)_j$ is \emph{non-basic variable}, and $c_i$ is a constant.
The basic variables constitute a \emph{basis}.
The basic and non-basic variables form a partition of the variables, and
the objective function is obtained by substituting the basic variables with non-basic variables.

\subsection{Parametric linear programming}
A parametric linear program (PLP) is a linear program, subjecting to $A\bm{\lambda} = \bm{b},\ \bm{\lambda} \geq 0$, whose objective function $Z(\bm{\lambda},\bm{x})$ contains parameters $\bm{x}$ appearing linearly.%
\footnote{There also exist parametric linear programs where the parameters are in the constant terms of the inequalities, we do not consider them here.}
%
%
The PLP reaches optimum at the vertex $\bm{\lambda}^*$, and the optimal solution is a set of $(\region[i], Z_i^*(\bm{x}))$. $\region[i]$ is the region of parameters $\bm{x}$, in which the basis does not change. $Z_i^*(\bm{x})$ is the optimal function corresponding to $\region[i]$, meaning that all the parameters in $\region[i]$ will lead to the same optimal function $Z_i^*(\bm{x})$. In the case of \emph{primal degeneracy} (\seci{\ref{sec:deg}}), the optimal vertex $\bm{\lambda}^*$ has multiple partitions of basic and non-basic variables, thus an optimal function can be obtained by different bases, i.e., several regions share the same optimal function.

\subsection{Redundant constraints}
\begin{definition}[Redundant]
A constraint is said to be redundant if it can be removed without changing the shape of the polyhedron.
\end{definition}

In our algorithms, there are several steps at which redundant constraints must be removed, which we call \textit{minimization} of the polyhedron. For instance we have $P=\{C_1:x_1 - 2x_2 \leq -2, C_2: -2x_1 + x_2 \leq -1, C_3: x_1 + x_2 \leq 8, C_4: -2x_1 - 4x_2 \leq -7\}$, and $C_4$ is a redundant constraint.

%

The redundancy can be tested by Farkas' Lemma: a redundant constraint can be expressed as the combination of some other constraints.
\begin{theorem}[Farkas' Lemma]
	Let $A \in \mathbb {R} ^{m\times n} A \in \mathbb {R} ^{m\times n}$ and $\bm{b} \in \mathbb {R} ^{m} \bm{b} \in \mathbb {R} ^{m}$. Then exactly one of the following two statements is true:
	\begin{itemize}
	\item There exists an  $\bm{x} \in \mathbb {R} ^{n}$ such that $A\bm{x}= \bm{b}$  and $\bm{x} \geq 0$.
	\item  There exists a $\bm {y} \in \mathbb {R} ^{m}$ such that ${A}^{\mathsf {T}}\bm{y} \geq 0$ and $\bm{b}^{\mathsf {T}}\bm{y} <0$.
	\end{itemize}
\end{theorem}

It is easy to determine the redundant constraints using Farkas' lemma, but in our case we have much more irredundant constraints than redundant ones, in which case using Farkas' lemma is not efficient. A new minimization algorithm which can find out the irredundant constraints more efficiently is explained in \cite{marechal_2017_vmcai}.


\section{Algorithm}
\label{sec:algo}

As our PLP algorithm is implemented with mix of rational numbers and floating-point numbers, we will make explicit the type of data used in the algorithm.
In the pseudo-code, we annotate data with $(name^{type})$, where $name$ is the name of data and $type$ is either \rational~or/and \float.
\bothfields~ means that the data is stored in both rational and floating-point numbers.

Floating-point computations are imprecise, and thus the floating-point LP solver may provide an incorrect answer: it may report that the problem is infeasible whereas it is feasible, that it is feasible even though it is infeasible, and it may provide an ``optimal'' solution that is not truly optimal.
What our approach guarantees is that, whatever the errors committed by the floating-point LP solvers, the polyhedron that we computed is a valid over-approximation: it always includes the polyhedron that should have been computed.
Details will be explained later in this section and in \seci{\ref{sec:checker}}.

In this section we do not consider the \emph{degeneracy}, which will be talked in \seci{\ref{sec:deg}}.

\subsection{Flow chart}
The \fig{\ref{fig:flow_chart}} shows the flow chart of our algorithm. The rectangles are processes and diamonds are decisions. The processes/decisions colored by orange are computed by floating-point arithmetic, and that by green uses rational numbers. The dotted red frames show the cases that rarely happen, which means that most computation in our approach uses floating-point numbers.

In \seci{\ref{sec:algo}} we will present the overview of the algorithm.
Then we will explain into details the processes/decisions framed by dashed blue rectangles in \seci{\ref{sec:checker}}.

\begin{figure}[!htb]
	\centering
	\includegraphics[width=0.95\textwidth, height=13cm]{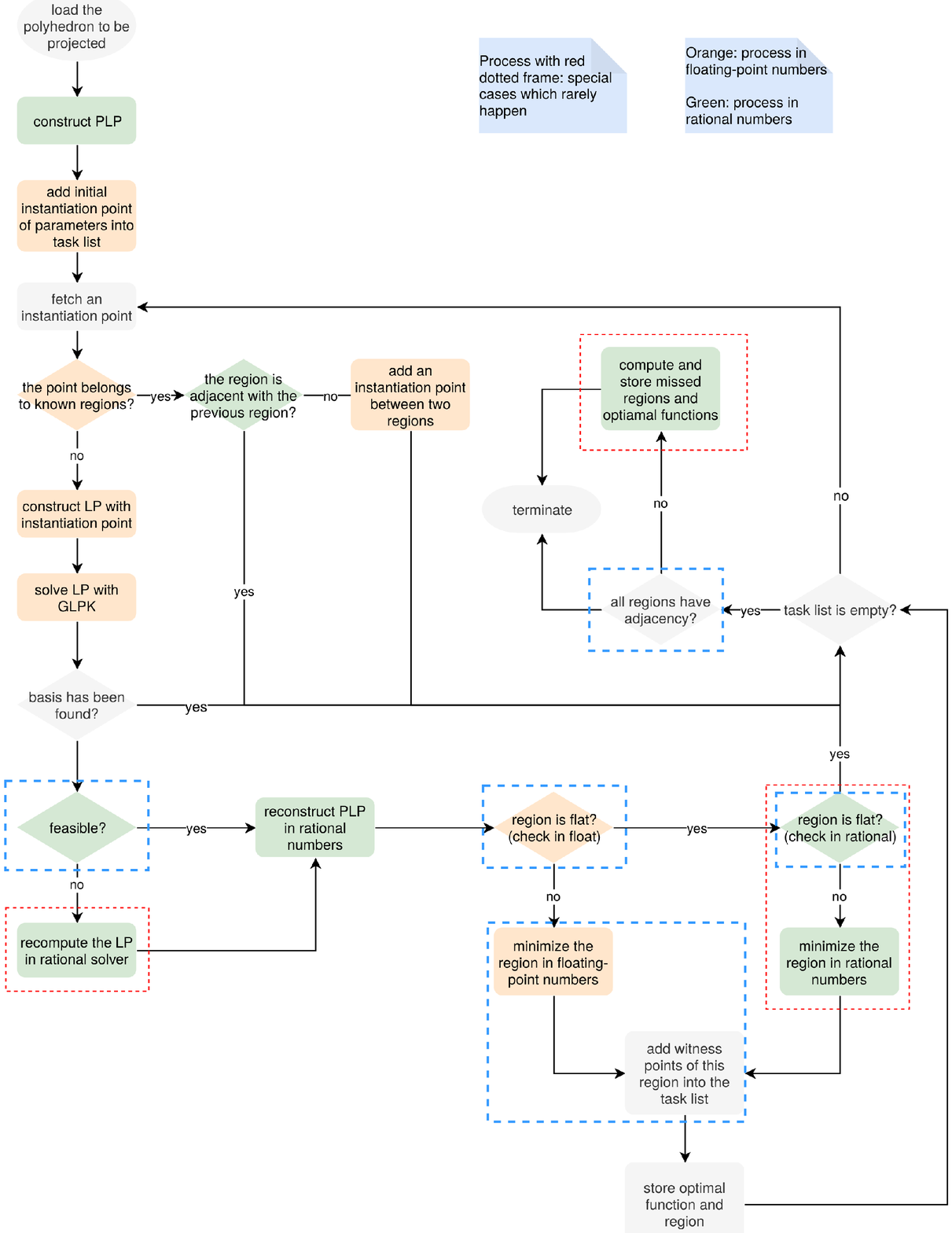}
	\caption{Flow chart}
	\label{fig:flow_chart}
\end{figure}

\subsection{Ray-tracing minimization}
At several steps we need to remove redundant constraints from the description of a polyhedron. We here present an efficient ray-tracing minimization method based on \cite{marechal_2017_vmcai}.
Their approach used rational computations, while ours uses floating-point arithmetic. The use of floating-point numbers here will not cause a soundness problem: in the worst case, we will eliminate constraints that should not be removed. In other words, when the floating-point algorithm cannot determine the redundancy, the corresponding constraints will be reported as redundant.

There are two phases in ray-tracing minimization. In the first phase we launch rays to the constraints, and the first hit constraints are irredundant. The remaining constraints will be determined in the second phase: if we can find the \textit{irredundancy witness} point, then the constraint is irredundant. The algorithm is shown in \algo{\ref{algo:minimization}}.
\begin{definition}[Irredundancy Witness]
	The irredundancy witness of a constraint $\cons[i]$ is a point that violates $\cons[i]$ but satisfies the other constraints.
\end{definition}
\begin{algorithm}[!htb]
	\DontPrintSemicolon
	\SetNoFillComment
	\KwIn{\var{\float[poly]}: the polyhedron to be minimized}
	\KwOut{the index of the irredundant constraints}
	\Fn{Minimize(\var{\float[poly]})}{
		\var{\float[p]} = GetInternalPoint(\var{\float[poly]}) \\
		\var{\float[rays]} = LaunchRays({\var{\float[poly]}, \var{\float[p]}}) \\
		\ForEach{\var{\float[ray]} in \var{\float[rays]}} {
			\var{constraintIdx} = FirstHitConstraint(\var{\float[poly]}, \var{\float[ray]}, \var{\float[p]})
			SetAsIrredundant(\var{\float[poly]}, \var{constraintIdx})
		}
		\ForEach{constraint \var{idx} in undetermined constraints} {
			\uIf{cannot determine} {
				SetAsRedundant(\var{\float[poly]}, \var{idx})\\
			}
			\Else {
				\uIf{found irredundancy witness point} {
					SetAsIrredundant(\var{\float[poly]}, \var{idx})
				}
				\Else{SetAsRedundant(\var{\float[poly]}, \var{idx})}
			}
		}
		\Return{the irredundant constraints}
	}
	\caption{Ray-tracing minimization algorithm.}
	\label{algo:minimization}
\end{algorithm}

\subsection{Parametric linear programming solver}
The algorithm is shown in \algo{\ref{algo:plp_serial}}. Firstly we construct the PLP problem, and then we solve it by solving a set of LP problems via floating-point LP solver. Then the rational solution will be reconstructed based on the information obtained from the LP solver.
We will explain each step in the following sections. Our focus will be on the cooperation of rational and floating-point numbers, and the tricks for dealing with floating-point arithmetic.

\begin{algorithm}[!htb]
	\DontPrintSemicolon
	\SetNoFillComment
	\KwIn{\var{\rational[poly]}: the polyhedron to be projected}
	\otherinput{$[x_p,...,x_q]$: the variables to be eliminated}
	\otherinput{\var{n}: number of initial points}
	\KwOut{\var{\rational[optimums]} the set of optimal function}
	\otheroutput{\var{\bothfields[regions]} the corresponding regions}
	\Fn{Plp(\var{\rational[poly]}, $[x_p,...,x_q]$, \var{n})}{
		\var{\bothfields[plp]} = ConstructPlp(\var{\rational[poly]}, $[x_p,...,x_q]$) \\
		\var{\float[worklist]} = GetInitialPoints(\var{\rational[poly], n}) \\
		\var{\rational[optimums]} = \none \\
		\var{\bothfields[regions]} = \none \\
		\While{\var{\float[worklist]} $\neq$ \none }{
			(\var{\float[w], \rational[R_{from}], $F_{from}$}) = getTask(\var{\float[worklist]}) \\
			\var{\rational[R_{curr}]} = CheckCovered(\var{\float[regions], \float[w]}) \\
			\If{\var{\rational[R_{curr}]} == \none}{
				(\var{basicIndices, nonbasicIndices}) = GlpkSolveLp(\var{\float[w], \float[plp]}) \\
				\var{\rational[reconstructMatrix]} = Reconstruct(\var{\rational[plp], basicIndices}) \\
				\var{(\rational[newOptimum], \bothfields[newRegion])} = ExtractResult(\var{\rational[reconstructMatrix], nonbasicIndices}) \\
				(\var{activeIndices},~\float[witnessList]) = Minimize(\var{\float[newRegion]}) \\
				\var{\rational[minimizedR]} = GetRational(\var{\rational[newRegion]}, activeIndices) \\
				Insert(\var{\rational[optimums]}, \var{\rational[newOptimum]}) \\
				Insert(\var{\rational[regions]}, \var{\rational[newRegion]}) \\
				AddWitnessPoints(\var{\float[witnessList], worklist}) \\
				\var{\rational[R_{curr}]}=\var{\rational[minimizedR]}
			}
			\uIf{Adjacent(\var{\rational[R_{curr}], \rational[R_{from}], $F_{from}$})}{
				$F_{curr}$ = GetCrossFrontier(\var{\rational[R_{curr}], \rational[R_{from}], $F_{from}$}) \\
				StoreAdjacencyInfo(\var{\rational[R_{from}], $F_{from}$, \rational[R_{curr}], $F_{curr}$})
			}
			\Else{
				AddExtraPoint(\var{worklist, \var{\rational[R_{curr}], \rational[R_{from}]}})
			}
			
		}
	}
	\caption{Parametric linear programming algorithm.}
	\label{algo:plp_serial}
\end{algorithm}

\subsubsection{Constructing PLP for projection}
The polyhedron to be projected is $\poly$: $A\bm{x} + \bm{b} \geq 0$. To perform projection, we can construct a PLP problem shown in \prob{\ref{problem:plp}}. In this problem, $\bm{x}$ are parameters, and $\bm{\lambda}$ are decision variables, where $\bm{x}=[x_1, \cdots, x_m]^\mathsf{T}$, $\bm{\lambda}=[\lambda_0, \cdots, \lambda_n]^\mathsf{T}$. Assume that we wish to eliminate $x_p,\cdots,x_q$, where $1 \leq p \leq q \leq m$.

\begin{align}
	\begin{aligned}
		\text{minimize} \quad &\sum_{i=1}^{n} (\bm{a}_{i\bullet} \bm{x} + b_i)\lambda_i + \lambda_0 \\
		\text{subject to} \quad &\sum_{i=1}^{n} (\bm{a}_{i\bullet} \bm{p} + b_i)\lambda_i + \lambda_0 = 1  \quad \text{(*)} \\
		&\sum_{i=1}^{n} a_{ij}\lambda_i = 0 \quad (\forall j \in \{p, \cdots, q\}) \quad \text{(**)} \\
		\text{and} \quad &\lambda_i \geq 0 \quad (\forall i \in \{0, \cdots, n\} ) \\
	\end{aligned}
	\label{problem:plp}
\end{align}
where $\bm{p}=[p_1,\cdots,p_m]$ is a point inside $\poly$. The constraint $(*)$ is called normalization constraint.
To compute the convex hull of $\poly$ and $\poly'$: $A'\bm{x} + \bm{b}' \geq 0$, we just replace the constraints $(**)$ with $A^T\bm{\lambda} - A'^T\bm{\lambda}' = 0,\ \bm{b}^T\bm{\lambda}+\lambda_0 - \bm{b}'^T\bm{\lambda}'-\lambda_0'=0$.
For more details about constructing the PLP problem of projection, please refer to \cite{jones_2008_polyhedral,marechal_2017_sas}.

\subsubsection{Solving PLP}
The PLP problem represents a set of LP problems, whose constraints are the same and objective function varies with the instantiation of the parameters.
Here is a brief sketch of our solver.
We maintain a working set of tasks yet to be performed.
At the beginning, a random vector of parameters (or a fixed one) is chosen as the initial task to trigger the algorithm.
Then, as long as the working set is not empty,
a vector of parameters $\bm{w}$ is taken from the working set.
We solve the (non-parametric) linear programming problem for this vector of parameters, using an off-the-shelf floating-point solver.
From the information of the final basis reached, we obtain a polyhedral region $\region$ of parameters, to which $\bm{w}$ belongs, that all share the same optimum and the same basis, as it will be explained below.
In general, this region is obtained with redundant constraints, so we minimize its representation.
The witness points $\bm{w}_1,\dots,\bm{w}_m$ of the irredundant constraints lie outside of $\region$, and are inserted into the working set.
We also maintain a set of already created regions: a vector $\bm{w}$ of parameters is ignored if it lies inside one of them.
The algorithm stops when the working set is empty, meaning that the full set of parameters is covered by regions.

Here is how we process a vector $\bm{w}$ from the working set.
We solve the LP problem:

\begin{align}
	\begin{aligned}
		\text{minimize} \quad &\sum_{i=1}^{n} (\bm{a}_{i\bullet} \bm{w} + b_i) \lambda_i + \lambda_0 \\
		\text{subject to} \quad &\sum_{i=1}^{n} (\bm{a}_{i\bullet} \bm{p} + b_i) \lambda_i + \lambda_0 = 1  \quad \text{(*)} \\
		&\sum_{i=1}^{n} a_{ij}\lambda_i = 0 \quad (\forall j \in \{p, \cdots, q\}) \\
		\text{and} \quad &\lambda_i \geq 0 \quad (\forall i \in \{0, \cdots, n\} ) \\
	\end{aligned}
	\label{problem:lp}
\end{align}

\subsubsection{Obtaining rational solution}
We solve this LP problem in floating-point using \glpk.
Had the solving been done in exact arithmetic, one could retain the optimal point $\bm{\lambda}^*$, but here we cannot use it directly.
Instead, we obtain the final partition of the variables into basic and non-basic variables, and from this partition we can recompute exactly, in rational numbers, the optimum $\bm{\lambda}^*$, as well as a certificate that it is feasible.

Let $M$ denote the matrix of constraints and $O$ that of the PLP objective function. The last column of the each matrix represents the constant.

\begin{equation}
M = 
\begin{bmatrix}
(A\bm{p+b})^\mathsf{T} & 1 & 1 \\
(\bm{\bm{a}_{\bullet p}})^\mathsf{T} &0 &0 \\
\vdots & \vdots & \vdots \\
(\bm{\bm{a}_{\bullet q}})^\mathsf{T} &0 &0 \\
\end{bmatrix}
\qquad
O =
\begin{bmatrix}
A^\mathsf{T} & 0 & 0 \\
\bm{b}^\mathsf{T} & 1 & 0
\end{bmatrix}
\end{equation}

To generate the result of PLP, we need to reconstruct the matrices $M$ and $O$ to make sure the objective function of PLP contains the same basis as the final tableau of the simplex algorithm: the coefficients of the basic variables in the objective function should be 0.
We extract the indices of the basic variables from that tableau; $M_B$ and $O_B$ denote the sub-matrices from $M$ and $B$ containing only the columns corresponding to the basic variables.
By linear algebra in rational arithmetic~%
\footnote{We use Flint, which provides exact rational scalar, vector and matrix computations, including solving of linear systems. \url{http://www.flintlib.org/}}
we compute a matrix $\Theta$, representing the substitution performed by the simplex algorithm. Then we apply this substitution to the objective matrix $O$ to get the new objective function $O'$: $\Theta = O_B M_B^{-1},\ O' = O - \Theta M $, where $M_B^{-1}$ denotes the inverse of $M_B$ (actually, we do not inverse that matrix but instead call a solver for systems of linear equations).

In our LP problem \ref{problem:lp}, the variables $\bm{\lambda}$ have lower bound 0, which means that when the objective function reaches the optimal, all the non-basic variables should reach their lower bound and their coefficients should be non-negative, otherwise the optimal value can decrease furthermore.
The same applies to the parametric linear problems, except that the coefficients of the objective function may contain parameters;
thus the sign conditions on these coefficients is translated to linear inequalities on these parameters.
Each non-zero column in $O'$ represents a function in $\bm{x}$, which is the coefficient of a non-basic variable. The conjunction of constraints $(O'_{\bullet j})^\mathsf{T} \bm{x} \geq 0$ constitute the region of $\bm{x}$ where $j$ belongs to the indices of non-basic variables.
This conjunction of constraints may be redundant: we thus call the minimization procedure over it.

\section{Checkers and rational solvers}
\label{sec:checker}

We compared our results with those from NewPolka.
We tested about 1.75 million polyhedra in our benchmarks.
In only 3 cases, round-off errors caused 1 face being missed.
In this section, we explain how we modified our algorithm to work around this difficulty.
The resulting implementation then computes exactly solutions to parametric linear programs, and thus exactly the same polyhedra as NewPolka.

\subsection{Verifying feasibility of the result from GLPK}
\glpk uses a threshold ($10^{-7}$ by default) to check feasibility, that is, if the solution it proposes truly is a solution.
It may report a feasible result when the problem is in fact infeasible. Assume that we have an LP problem whose constraints are $C_1: \lambda_1 \geq 0, C_2: \lambda_2 \geq 0, C_3: \lambda_1+\lambda_2 \leq 10^{-8}$, \glpk will return $(0,0)$ as a solution, whereas it is not.

We use \flint to compute the row echelon form of the rational matrix of constraints, so that the pivots are the coefficients of basic variables.
We obtain $[I\ A'] = [\bm{b}]$~\footnote{There may be rows of all zeros in the bottom of the matrix.}, where $A'$ are the coefficients of the non-basic variables.
When the LP problem reaches an optimum, the non-basic variables are at their lower bound 0, so the value of the basic variables are just the value of $\bm{b}$.
As we have the constraints that the variables are non-negative, we thus just need to verify that all coordinates in $\bm{b}$ are non-negative.
If it is not in this case, it means that \glpk does not have enough precision, which is likely due to an ill-conditioned subproblem.
In this case, we start a textbook implementation of the simplex algorithm in rational arithmetic.

\glpk may also report an optimal solution which is in fact not optimized. We did not provide a checker for this situation, as even if the solution is not optimized in the required region, it is optimized in anther region which is probably adjacent to the expected one. We keep the obtained solution, and add extra task points between the regions if they are not adjacent. Besides the adjacency checker guarantees there will be no missed face. 

\subsection{Flat regions}
Our regions are obtained from the rational matrix, and then they are converted into floating-point representation.
As the regions are normalized and intersect at the same point, they are in the shape of cones.
During the conversion, the constrains will lose accuracy, and thus a cone could be misjudged as flat, meaning it has empty interior.
For instance, we have a cone $\{ \cons[1]: - \frac{100000001}{10000000} x_1 + x_2 \leq 0, \cons[2]: \frac{100000000}{10000000} x_1 - x_2 \leq 0 \}$, which is not flat. After conversion, $\cons[1]$ and $\cons[2]$ will be represented in floating-point numbers as $\{ \cons[1]: - 10.0 x_1 + x_2 \leq 0, \cons[2]: 10.0 x_1 - x_2 \leq 0 \}$, and the floating-point cone is flat.

In this case we invoke a rational simplex solver to check the region by shifting all the constraints to the interior direction. If the region becomes infeasible after shifting, then the region is really flat; otherwise we launch a rational minimization algorithm, which is implemented using Farkas’ Lemma, to obtain the minimized region.

\subsection{Computing an irredundancy witness point}
In the minimization algorithm, the checker makes sure that the constraints which cannot be determined by floating-point algorithm will be regarded as redundant constraints. In the meantime these constraints are marked as uncertainty. If the polyhedron to be minimized is also represented by rational numbers, a rational solver will be launched to determine the uncertain constraints. As in our PLP algorithm all the regions are represented by both floating-point and rational numbers, the rational solver can always be executed when there are uncertain constrains.

Consider the case of computing the irredundant witness point of the constraint $\cons[i]$, we need to solve a feasibility problem: $\cons[i]: \bm{a_ix} < b_i$ and $\cons[j]: \bm{a_jx} \leq b_j, \forall j\neq i$.
For efficiency, we solve this problem in floating point.
However, \glpk does not support strict inequalities, thus we need tricks to deal with them.

One method is to shift the inequality constraint a little and obtain a non-strict inequality $\cons[1]': \bm{a_1x} \leq b_1 - \epsilon$, where $\epsilon$ is a positive constant.
This method is however difficult to apply properly because of the need to find a suitable $\epsilon$.
If $\epsilon$ is small, we are likely to obtain a point too close to the constraint $\cons[1]$;
if $\epsilon$ is too large, perhaps we cannot find any point. One exception is that when the polyhedron is a cone, we can always find a satisfiable point by shifting the constraints, no matter how large $\epsilon$ is.

We thus adopted another method for non-conic polyhedra.
Instead of solving a satisfiability problem, we solve an optimization problem:
\begin{align}
	\begin{aligned}
		\text{maximize} \quad &-\bm{a_ix} \\
		\text{subject to} \quad &\bm{a_jx} \leq b_j \quad \forall j \neq i\\
		&\bm{a_ix} \leq b_i\\
	\end{aligned}
\end{align}
The found optimal vertex is the solution we are looking for.

Assuming we have the polyhedron: $- x_1 + x_2 \leq 0, x_1 + x_2 \leq 7, -2x_2 < -3$.
The two methods are shown in \fig{\ref{fig:float_lp}}.
If we compute the optimum in the direction $x_2$ with constraints $- x_1 + x_2 \leq 0, x_1 + x_2 \leq 7$, we obtain a feasible point $(3.5,3.5)$.

\begin{figure}[!htb]
	\centering
	\includegraphics[width=0.45\textwidth, height=3cm]{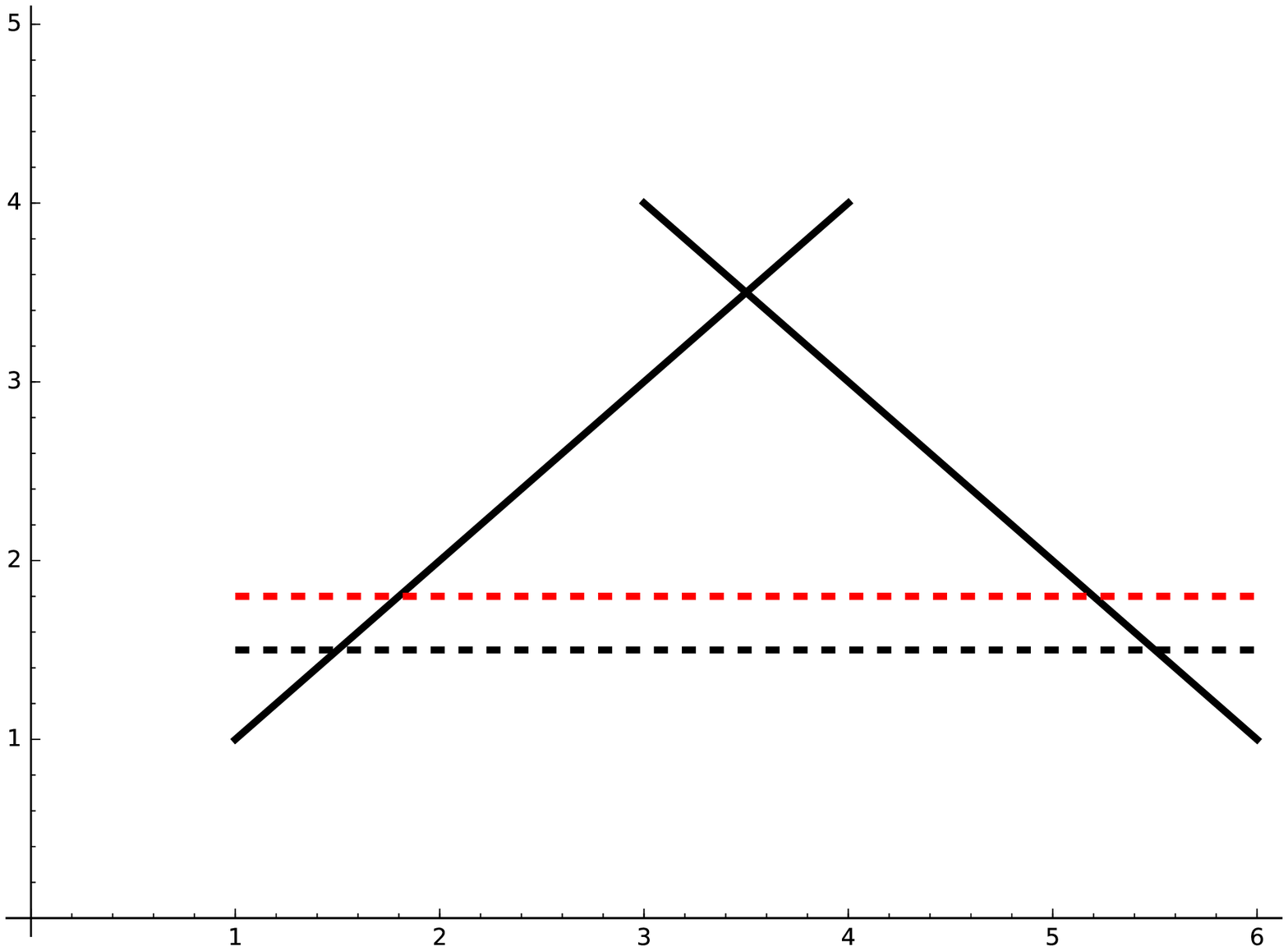}
	\includegraphics[width=0.45\textwidth, height=3cm]{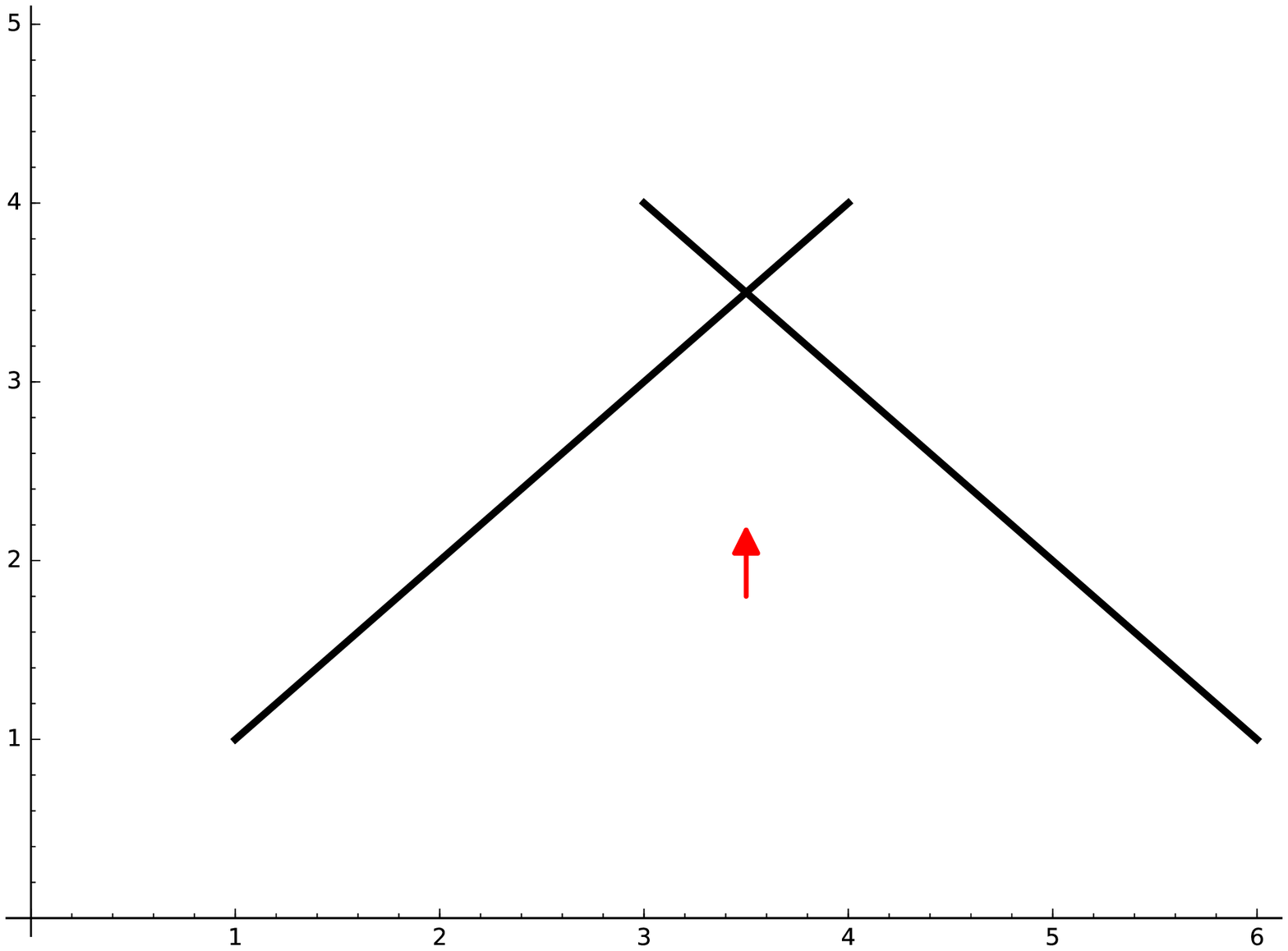}
	\caption{Solving an optimization problem instead of a feasibility problem.}
	\label{fig:float_lp}
\end{figure}

However the floating-point solver could misjudge, thus the found optimal vertex $\bm{p}$ could be infeasible. Hence we need to test $\bm{a_ip} \leq b_i - t$, where $t$ is the \glpk threshold. If the test fails, we will use the rational simplex algorithm to compute the Farkas combination: the constraint is really irredundant if the combination does not exist.

\subsection{Adjacency checker}
\label{sec:adj_checker}
We shall now prove that no face is missed if and only if for each region and each boundary of this region, another region is found which shares that boundary.

Assuming we have a situation shown in \fig{\ref{fig:faces}}: the four regions correspond to different optimal functions. $\region[1], \region[2]$ and $\region[3]$ all found their adjacencies, but $\region[4]$ is missed. In this case there exist two adjacent regions for some boundaries.
We here show that this situation will not happen.

\begin{theorem}
	No face will be missed if each region finds all the adjacent regions.
\end{theorem}
\begin{proof}
Assume that we cross the boundary $\frontier$ of the region $\region[i]$, and the adjacent regions are $\region[j]$ and $\region[k]$. The corresponding optimal functions are $Z_j$ and $Z_k$, and $Z_j \neq Z_k$ (otherwise no face will be
missed). From $\region[i]$ to its adjacency, we need to do one pivoting. Consider the simplex tableau in Table
\ref{tableau: proof}. Assuming the entering variable is $\lambda_q$. If there are two adjacent regions, there will be two
possible leaving variables, say $\lambda_r$ and $\lambda_s$. In the simplex algorithm we always choose the variable with
the smallest ratio of the constant and the coefficient as the leaving variable. When there are two possible leaving
variables, the value of these two ratios must be equal, that is $\frac{b_j}{a_{jq}} = \frac{b_k}{a_{kq}}$. In this case
we face the primal degeneracy, and $f^*(\bm{x})-\frac{b_j}{a_{jq}} f_q = f^*(\bm{x})-\frac{b_k}{a_{kq}} f_q$. This is a contradictory to the assumption $Z_j \neq Z_k$. Hence the situation will not happen.
	\qed
\end{proof}

\begin{table}
	\begin{tabular}{c|ccccccccccc}
		& \multicolumn{4}{c}{$\overbrace{\rule{3.5cm}{0pt}}^\text{non-basic variables}$} & \multicolumn{6}{c}{$\overbrace{\rule{5cm}{0pt}}^\text{basic variables}$} & constants \\
		& $\lambda_1$ & $\cdots$ & $\lambda_q$ & $\cdots$ & $\cdots$ & $\lambda_r$ & $\cdots$ & $\lambda_s$ & $\cdots$ & $\lambda_n$ & \\
		\hline
		objective & $f_1$ & $\cdots$ & $f_q$ & $\cdots$ & $\cdots$ & 0 & $\cdots$ & 0 & $\cdots$ & 0 & $Z^*(\bm{x})$ \\
		\hline
		$\vdots$ & \multicolumn{11}{ p{11cm} }{\dotfill} \\
		row $j$ & \multicolumn{2}{ p{2cm} }{\dotfill} & $m_{jq}$ & $\cdots$ & $\cdots$ & 1 & $\cdots$ & 0 & $\cdots$ & 0 & $c_j$ \\
		$\vdots$ & \multicolumn{11}{ p{11cm} }{\dotfill} \\
		row $k$ & \multicolumn{2}{ p{2cm} }{\dotfill} & $m_{kq}$ & $\cdots$ & $\cdots$ & 0 & $\cdots$ & 1 & $\cdots$ & 0 & $c_k$ \\
		$\vdots$ & \multicolumn{11}{ p{11cm} }{\dotfill}
	\end{tabular}
	\caption{Simplex tableau.}
	\label{tableau: proof}
\end{table}

\begin{figure}[!htb]
	\centering
	\includegraphics[width=0.45\textwidth, height=3cm]{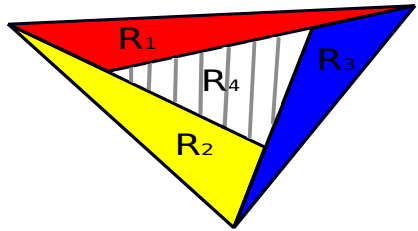}
	\caption{Example of missing faces.}
	\label{fig:faces}
\end{figure}

To find out all the faces, we just need to ensure that all the regions have their adjacencies. Although we tried to add task points between the regions which are not adjacent, there may be still missed region because of floating-point arithmetic. Hence we invoke an adjacency checker at the end of the algorithm. The information of adjacency has been saved in \algo{\ref{algo:plp_serial}}: if the regions $\region[i]$ and $\region[j]$ are adjacent by crossing the boundaries $\frontier[m]$ and $\frontier[n]$, we set true to $(\region[i], \frontier[m])$ and $(\region[j], \frontier[n])$ in the adjacency table. The checker will find out the pair $(\region[k], \frontier[p])$ whose flag of adjacency is false. Then we cross the boundary $\frontier[p]$ and use \algo{\ref{algo:degeneracy}} to compute the missed region and the corresponding optimal function. The adjacencies of the new obtained region will be checked then, and the algorithm terminates when all the obtained regions have complete adjacencies.

\section{Overlapping regions and degeneracy}
\label{sec:deg}
Ideally, the parametric linear programming outputs a quasi-partition of the space of parameters, meaning that the produced regions do not have overlap except at their boundary (we shall from now on say ``do not overlap'' for short) and cover the full space of parameters.
This may not be the case due to two reasons:
geometric degeneracy, leading to overlapping regions, and imprecision due to floating-point arithmetic, leading to insufficient coverage. The latter will be dealt with by rational checker, which has been explained in \seci{\ref{sec:checker}}.

If regions do not overlap, it is possible to verify that the space of parameters is fully covered by checking that each boundary of a region is also a boundary of an adjacent region (proof in \seci{\ref{sec:adj_checker}}); otherwise, this means we have a boundary with nothing on the other side, thus the space is not fully covered.
This simple test does not work if regions overlap.
Furthermore, overlapping regions may be needlessly more numerous than those in a quasi-partition.
We thus have two reasons to modify our algorithm to get rid of overlapping regions.

Let us see how overlapping regions occur.
In a non-degenerate parametric linear program, for a given optimization function, there is only one optimal vertex (no \emph{dual degeneracy}), and this optimal vertex is described by only one optimal basis (no \emph{primal degeneracy}), i.e., there is a single optimal partition of variables into basic and non-basic.
Thus, in a non-degenerate parametric linear program, for a given vector of parameters there is one single optimal basis (except at boundaries), meaning that each optimal function corresponds to one region.
However when there is degeneracy, there will be multiple bases corresponding to one optimal function, and each of them computes a region. These regions may be overlapping. We call the regions corresponds to the same optimal function \emph{degeneracy regions}.

\begin{theorem}
	There will be no overlapping regions if there is no degeneracy.
\end{theorem}

\begin{proof}
	In parametric linear programming, the regions are yielded by the partition of variables into basic and non-basic, i.e., each region corresponds to one basis.
	The parameters within one region lead the PLP problem to the same partition of variables.
	If there are overlapping regions, say $\region[i]$ and $\region[j]$, the PLP problem will be optimized by multiple bases when the parameters belong to $\region[i] \cap \region[j]$. 
	In this case there must be degeneracy:
	these multiple bases may lead to multiple optimal vertex when we have dual degeneracy, or the same optimal vertex when we have primal degeneracy.
	By transposition, we know that if there is no degeneracy the PLP problem will always obtain a unique basis with given parameters, and there will be no overlapping regions.
	\qed
\end{proof}

We thus need to get rid of degeneracy.
We shall first prove that there is no dual degeneracy in our PLP algorithm, and then deal with the primal degeneracy.

\subsection{Dual degeneracy}
\begin{theorem}
For projection and convex hull, the parametric linear program exhibits no dual degeneracy.
\end{theorem}
\begin{proof}
We shall see that the \emph{normalization constraint} (the constraint $(*)$ in \prob{\ref{problem:plp}}) present in the parametric linear programs defining projection and convex hull prevents dual degeneracy.

Assume that at the optimum $Z^*(\bm{x})$ we have the simplex tableau in Table \ref{tableau: proof}.  
$\lambda_k$ denote the decision variables: $\lambda_k \geq 0$.
In the current dictionary, the parametric coefficients of the objective function is $f_k = \bm{a}'_{i\bullet} \bm{x} + b'_i$. 
Assuming the variable leaving the basis is $\lambda_r$, and the entering variable is $\lambda_q$. Then $\lambda_r$ is defined by the $j$th row as $\sum_j m_{jp}\lambda_p + \lambda_r = c_j$, where $\lambda_p$ are nonbasic variables. That means $\lambda_r = c_j$ when the nonbasic variables reach their lower bound, which is 0 here.

Now we look for another optimum by doing one pivoting. As the current dictionary is feasible, we must have $c_j \geq 0$. To maintain the feasibility, we must choose $\lambda_q$ such that $m_{jq} > 0$. 
As we only choose the non-basic variable whose coefficient is negative to enter the basis, then we know $f_q < 0$. By pivoting we obtain the new objective function $Z'(\bm{\lambda},\bm{x}) = Z(\bm{\lambda},\bm{x}) - f_q \frac{c_j}{m_{jq}}$. The new optimal function is: 
\begin{equation}
	Z^{*'}(\bm{x}) = Z^*(\bm{x}) - f_q \frac{c_j}{m_{jq}}
	\label{eq:first}
\end{equation}

Let us assume that a dual degeneracy occurs, which means that we obtain the same objective function after the pivoting, i.e., $Z^{*'}(\bm{x}) = t Z^*(\bm{x})$, where $t$ is a positive constant. Due to the normalization constraint at the point $\bm{x}_0$ enforcing $Z^{*'}(\bm{x}_0)=Z^{*}(\bm{x}_0)=1$, we have $t=1$. Hence we will obtain
\begin{equation}
	Z^{*'}(\bm{x}) = Z^*(\bm{x})
	\label{eq:second}
\end{equation}

Considering the equation $\ref{eq:first}$ and $\ref{eq:second}$ we obtain 
\begin{equation}
	f_q \frac{c_j}{m_{jq}} = 0
\end{equation}

Since $f_q \neq 0$, $c_j$ must equal to 0, which means that we in fact faced a primal degeneracy.

Let $D_1 = f_q \frac{c_j}{m_{jq}}$, where the subscript of $D_1$ denotes the first pivoting. As $c_j \geq 0, f_q <0$ and $m_{jq} > 0$, we know $D_1 \leq 0$. Similarly in each pivoting we have $D_i \leq 0$.

If we generalize the situation above to $N$ rounds of pivoting, we will obtain:
\begin{equation}
	Z^{*'}(\bm{x}) = Z^*(\bm{x}) - \sum_{i=1}^N D_i
	\label{eq:general}
\end{equation}
If there is dual degeneracy $Z^{*'}(\bm{x}) = Z^*(\bm{x})$, and then
\begin{equation}
	\sum_{i=1}^N D_i = 0
	\label{apd_eq:sum_p}
\end{equation}
As $\forall i, D_i \leq 0$, \equa{\ref{apd_eq:sum_p}} implies $\forall i, D_i = 0$, which is possible if and only if all the $c_j$ equal to 0. For the same reason as above, in this case we can only have primal degeneracy.
\end{proof}

\subsection{Primal degeneracy}
Many methods to deal with primal degeneracy in non-parametric linear programming are known \cite{bland_1977_new,dantzig_2006_linear,dantzig_1951_application}; fewer in parametric linear programming~\cite{jones_2007_lexicographic}.	
We implemented an approach to avoid overlapping regions based on the work of Jones et al.\ \cite{jones_2007_lexicographic}, which used the perturbation method~\cite{dantzig_2006_linear}.
The algorithm is shown in \algo{\ref{algo:degeneracy}}.
Once entering a new region, we check if there is primal degeneracy: it occurs when one or several basic variables equal zero.
In this case we will explore all \emph{degeneracy regions} for the same optimum, using, as explained below, a method avoiding overlaps.

Let us consider a projected polyhedra in 3 dimensions with primal degeneracy, because of which there are multiple regions corresponding to the same face.
\fig{\ref{fig:deg}} shows the 2D view of the face. The yellow and red triangles represent the intersection of the regions with their face.
\fig{\ref{fig:deg_3}} shows the disappoint case where the regions are overlapping. The reason is that when the parameters locate in the orange part, two different bases will lead the constructed LP problem to optimum. We aim to avoid the overlap and obtain the result either in \fig{\ref{fig:deg_1}} or in \fig{\ref{fig:deg_2}}.

\begin{figure}[!htb]
	\centering
	\begin{subfigure}{.3\textwidth}
		\centering
		\includegraphics[width=\textwidth]{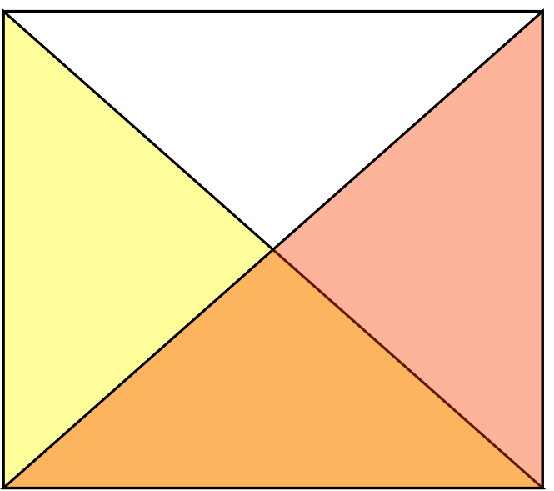}
		\caption{}
		\label{fig:deg_3}
	\end{subfigure}\hfil
	\begin{subfigure}{.3\textwidth}
		\centering
		\includegraphics[width=\textwidth]{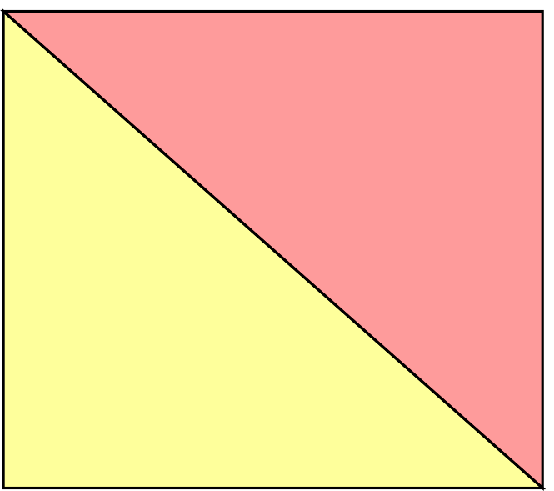}
		\caption{}
		\label{fig:deg_1}
	\end{subfigure}\hfil
	\begin{subfigure}{.3\textwidth}
		\centering
		\includegraphics[width=\textwidth]{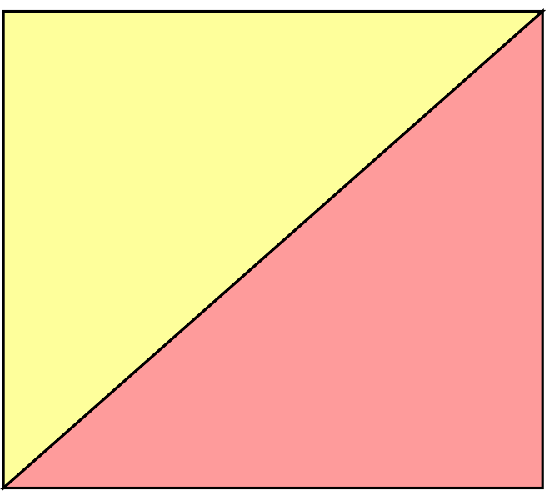}
		\caption{}
		\label{fig:deg_2}
	\end{subfigure}
	\caption{Example of overlapping regions.}
	\label{fig:deg}
\end{figure}

Our solution against overlaps is to make the optimal basis unique for given parameters of the objective function by adding perturbation terms to the right side of the constraints~\cite{jones_2007_lexicographic}.
These perturbation terms are ``infinitesimal'', meaning that the right-hand side, instead of being a vector of rational scalars, becomes a matrix where the first column corresponds to the original vector, the second column corresponds to the first infinitesimal, the third column to the second infinitesimal, etc.
The same applies to $\bm{\lambda}$.
Instead of comparing scalar coordinates using the usual ordering on rational numbers, we compare line vectors of rationals with the lexicographic ordering.
After the perturbation, there will be no primal degeneracy as all the right-hand side of the constraints cannot be equal.

The initial perturbation matrix is a $k*k$ identity matrix: $M_p = I$, where $k$ is the number of constraints. Then the perturbation matrix will be updated as the reconstruction of the constraint matrix.
After adding this perturbation matrix, the right-hand side becomes $B = [\bm{b}|M_p]$. 
The new constants are vectors in the form of $\bm{v_i} = [b_i\ 0\ \cdots\ 1\ \cdots\ 0]$.
We compare the vectors by lexico-order: $\bm{v_i} > \bm{v_j}$ if the first non-zero element of $\bm{v_i}$ is larger than that of $\bm{v_j}$.

To obtain a new basis, in contrast to working with non-degeneracy regions, we do not solve the problem using floating point solver.
Instead, we pivot directly on the perturbed rational matrix.
Each non-basic variable will be chosen as entering variable. Then from all the constraints in which $b_i=0$, we select the basic variable $\lambda_l$ in $C_i$ whose ratio $\frac{\bm{v_i}}{a_{ij}}$ is smallest as the leaving variable, where $j$ is the index of the entering variable. If such a leaving variable exist, we will obtain a degeneracy region: as $b_i=0$, the new optimal function will remain the same. Otherwise it means that a new optimal function will be obtained by crossing the corresponding frontier. The latter will not be treated by this algorithm, but will be computed with a task point by \algo{\ref{algo:plp_serial}}. We maintain a list of bases which have been explored. The algorithm terminates when all the degeneracy regions of the same optimal function are found.

\begin{algorithm}[!htb]
\DontPrintSemicolon
\SetNoFillComment
\KwIn{\var{\float[w]}: the task point}
\otherinput{\var{\rational[plp]}: the PLP problem to be solved}
\KwOut{degeneracy regions correspond to the same optimal solution}
\Fn{DiscoverNewRegion(\var{\float[w], \float[plp]})}{
	\var{basicIdx} = GlpkSolveLp(\var{\float[w], \float[plp]}) \\
	\If{degenerate}{
		\var{size} = GetSize(basicIdx) \\
		\var{\rational[perturbM]} = GetIdentityMatrix(size, size) \\
		\var{basisList} = \none \\
		Insert(\var{basisList}, \var{basicIdx}) \\
		\var{degBasic} = \none \\
		\ForEach{basic variable \var{v}}{
			\If{\var{v} == 0}{
				Insert(\var{degBasic}, GetIdx(\var{v})) \\
			}
		}
		\While{\var{basisList} $\neq$ \none}{
			\var{currBasis} = GetBasis(\var{basisList}) \\
			\If{\var{currBasis} has been found}{
				continue
			}
			\var{nonBasicIdx} = GetNonBasic(\var{currBasis}) \\
			(\var{\rational[reconstructM], \rational[perturbM]}) = Reconstruct(\var{\rational[plp], basicIdx, \rational[perturbM]}) \\
			\var{(\rational[newOptimum], \bothfields[newRegion])} =	 ExtractResult(\var{\rational[reconstructM], nonbasicIdx}) \\
			\var{activeIdx}=Minimize(\var{\float[newRegion]}) \\
			\var{\rational[minimizedR]} = GetRational(\var{\rational[newRegion]},  avtiveIdx) \\
			Insert(\var{\rational[optimums]}, \var{\rational[newOptimum]}) \\
			Insert(\var{\rational[regions]}, \var{\rational[newRegion]}) \\
			\ForEach{constraint \var{i} in \var{\rational[minimizedR]}}{
				\var{enteringV} = GetIdx(\var{i}) \\
				\var{leavingV} = SearchLeaving(\var{degBasic}, \rational[perturbM]) \\
				\If{\var{leavingV $\neq$ \none}}{
					\var{newBasis} = GetNewBasis(\var{basicIdx}, \var{enteringV}, \var{leavingV}) \\
					Insert(\var{basisList}, \var{newBasis})
				}
			}
		}
	}
}
\caption{Algorithm to avoid overlapping regions.}
\label{algo:degeneracy}
\end{algorithm}

\section{Experiments}
In this section, we analyze the performance of our parametric linear programming solver on projection operations.
We compare its performance with that of the NewPolka library of Apron%
\footnote{\url{https://github.com/antoinemine/apron}}
and ELINA library \cite{singh_2017_fast}.
Since NewPolka and ELINA do not exploit parallelism, we compare it to our library running with only one thread.

We used three libraries in our implementation:
\begin{itemize}
	\item Eigen 3.3.2 for floating-point vector and matrix operations;
	\item FLINT 2.5.2 for rational arithmetic, vector and matrix operations;
	\item GLPK 4.6.4 for solving linear programs in floating-point.
\end{itemize} 

The experiments are carried out on 2.30GHz Intel Core i5-6200U CPU.

\subsection{Experiments on random polyhedra}
\subsubsection{Benchmarks}

The benchmark contains randomly-generated polyhedra, in which the coefficients of constraints are in the range of -50 to 50. Each polyhedron has 4 parameters: number of constraints (CN), number of variables (VN), projection ratio(PR) and density (D).
The projection ratio is the proportion of eliminated variables: for example if we eliminate 6 variables out of 10, the projection ratio is 60\%.
Density represents the ratio of zero coefficients: if there are 2 zeros in 10 coefficients, density is 20\%.
In each experiment, we project 10 polyhedra generated with the same parameters.
To smooth out experimental noise, we do each experiment 5 times,
i.e., 50 executions for each set of parameters. Then we calculate the average execution time of the 50 executions. 

\subsubsection{Experimental results}
We illustrate the execution time (in seconds) by line charts. The blue line is the performance of NewPolka library of Apron, and the red line is that of our serial PLP algorithm. To illustrate the performance benefits from the floating-point arithmetic, we turned off \glpk and always use the rational LP solver, and the execution time is shown by the orange lines\footnote{The minimization is still computed in floating-point numbers.}. It is shown that solving the LP problems in floating-point numbers and reconstructing the rational simplex tableau leads to significant improvement of performance.

By a mount of experiments, we found that when the parameters $CN=19, VN=8, PR= 62.5\%$ and $D=37.5\%$, the execution time of PLP and Apron are similar, so we maintain three of them and vary the other to analyze the variation of performance.

Recall that in order to give a constraint description of the projection of a convex polyhedron $P$ in constraint description, Apron (and all libraries based on the same approach, including PPL) computes a generator description of $P$, projects it and then computes a minimized constraint description.

\paragraph{Projection ratio}
\label{sec:expr_proj}
In \fig{\ref{fig:exp_p_1}} we can see that execution time of PLP is almost the same for all the cases, whereas that of Apron changes significantly. Apron incurs a large cost when it computes the generator representation of each polyhedron. We plot the execution time of PLP (\fig{\ref{fig:exp_p_3}}) and the number of regions (\fig{\ref{fig:exp_p_4}}), which vary with the same trend. That means the cost of our approach depends mostly on the number of regions to be explored.
To illustrate it more clearly, the zoomed figure is shown in \fig{\ref{fig:exp_p_2}}.

The more variables are eliminated, the lower dimension the projected polyhedron has. Then the cost of chernikova's algorithm to convert from the generators into the constraints will be less. This explains why Apron is slow when the projection ratio is low, and becomes faster when the number of eliminated variable is larger.

\begin{figure}[!htb]
	\centering
	\begin{subfigure}{.45\textwidth}
		\centering
		\includegraphics[width=\textwidth, height=3.8cm]{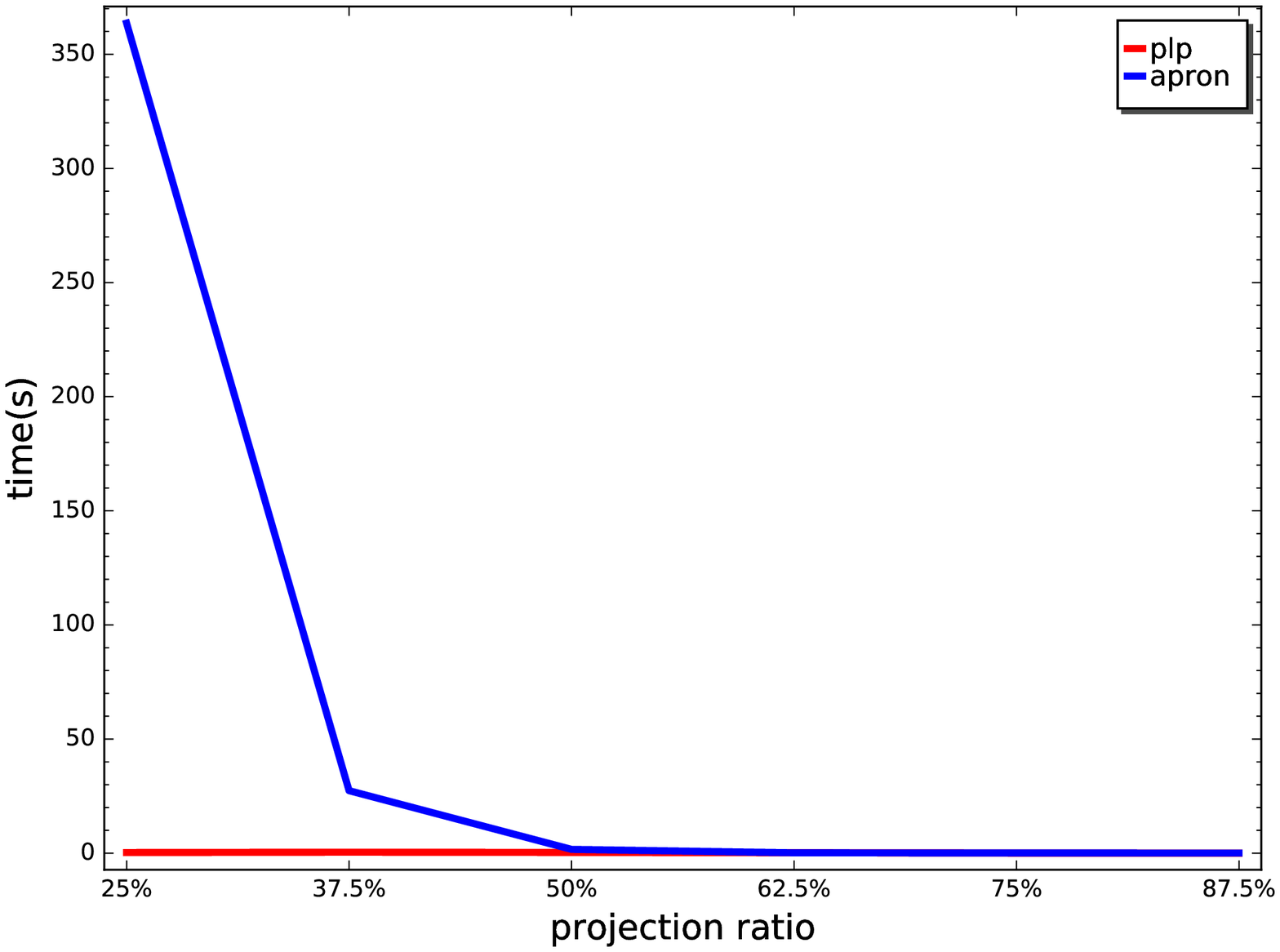}
		\caption{}
		\label{fig:exp_p_1}
	\end{subfigure}
	\begin{subfigure}{.45\textwidth}
		\centering
		\includegraphics[width=\textwidth, height=3.8cm]{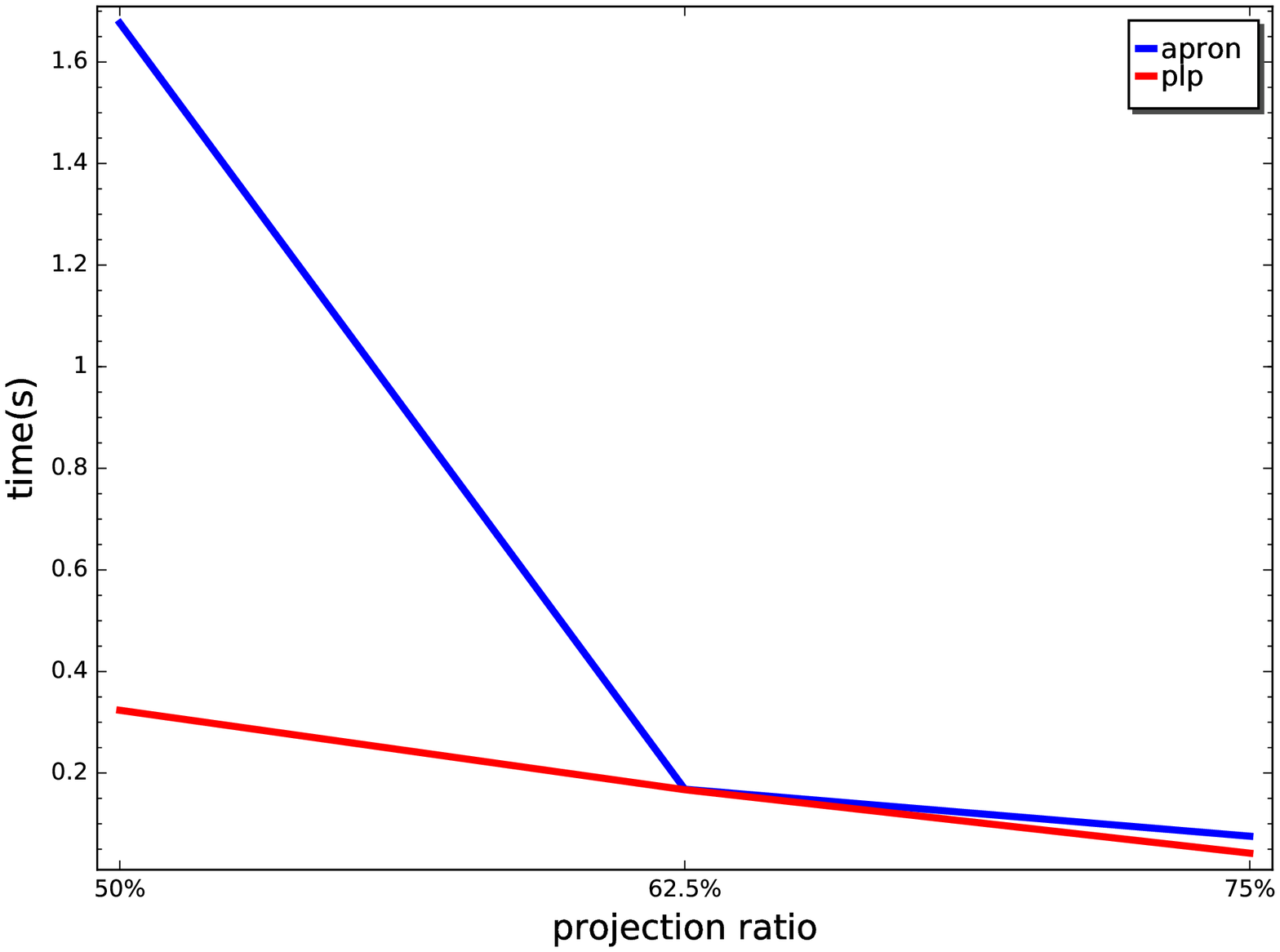}
		\caption{}
		\label{fig:exp_p_2}
	\end{subfigure}
	\begin{subfigure}{.45\textwidth}
		\centering
		\includegraphics[width=\textwidth,height=3.8cm]{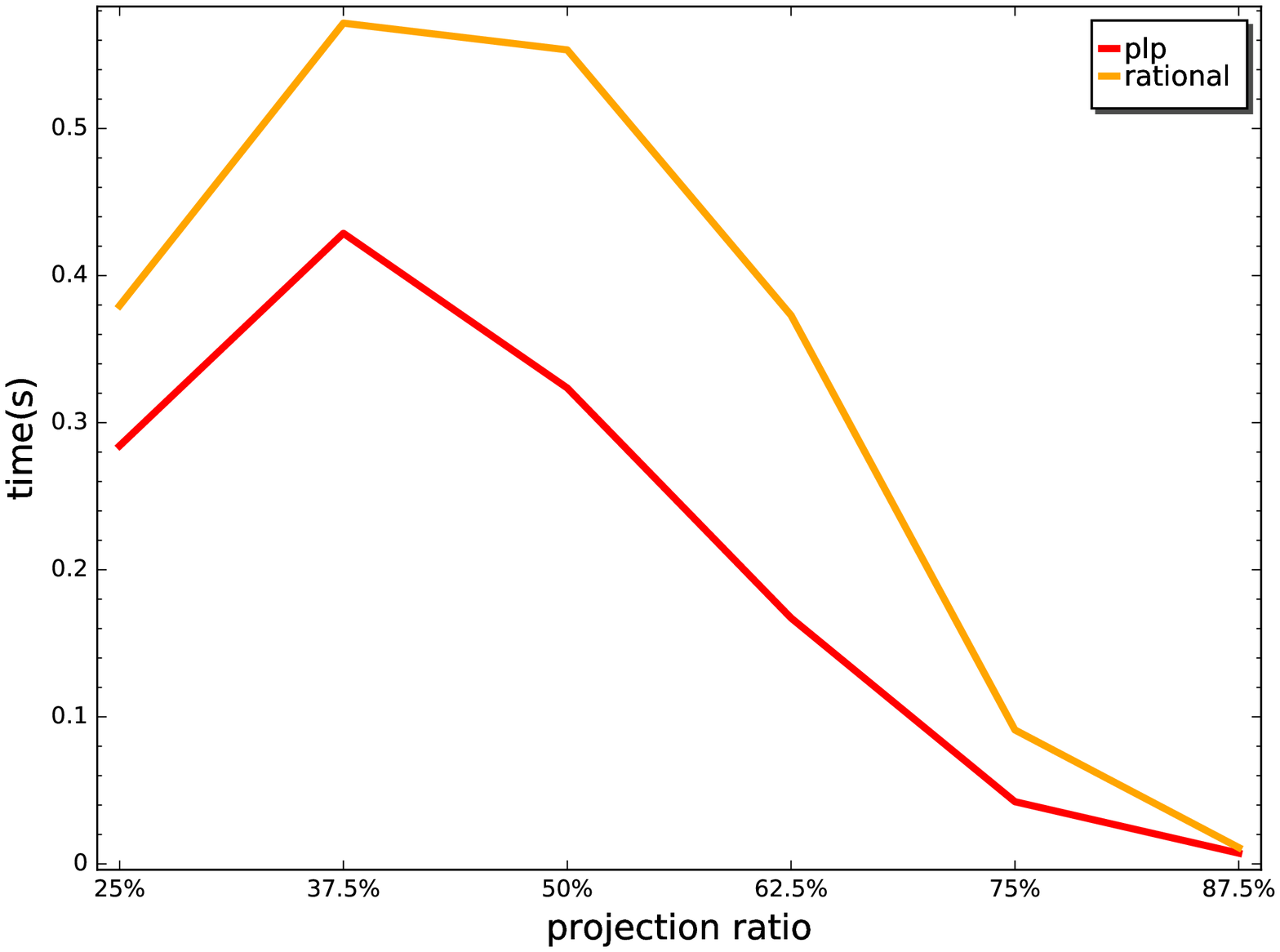}
		\caption{}
		\label{fig:exp_p_3}
	\end{subfigure}
	\begin{subfigure}{.45\textwidth}
		\centering
		\includegraphics[width=\textwidth,height=3.8cm]{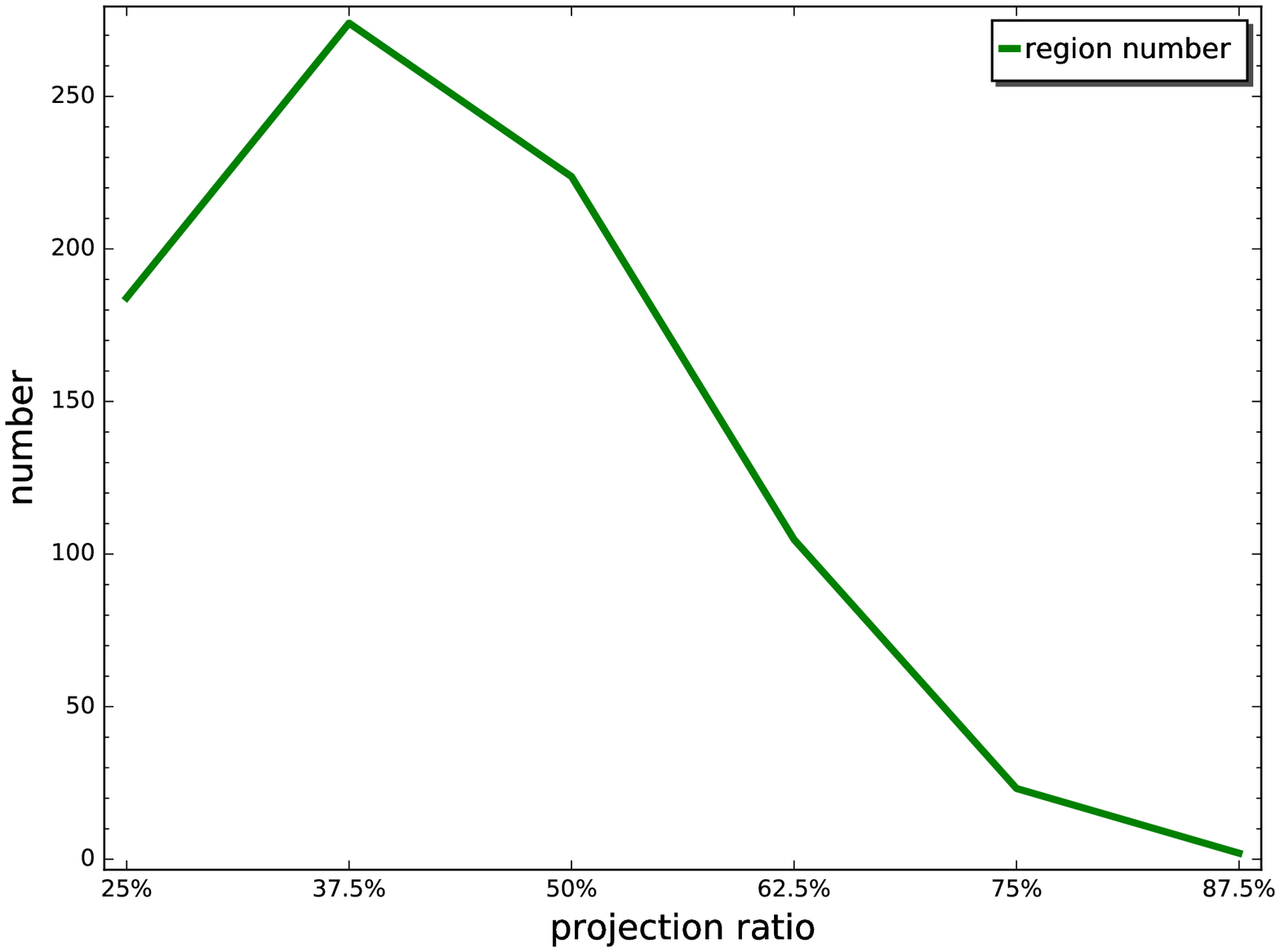}
		\caption{}
		\label{fig:exp_p_4}
	\end{subfigure}
	\caption{CN=19,VN=8,D=37.5\%,PR=[25\%,87.5\%]}
	\label{fig:exp_p}
\end{figure}

\paragraph{Number of constraints}
Keep the other parameters, we increase the number of constraints from 12 to 30. The result is shown in \fig{\ref{fig:exp_c}}.
We can see that Apron is faster than PLP when constraints are fewer than 19; beyond that, its execution time increases significantly. In contrast, the execution time of PLP grows much more slowly.

\begin{figure}[!htb]
	\centering
	\includegraphics[width=0.45\textwidth,height=3.8cm]{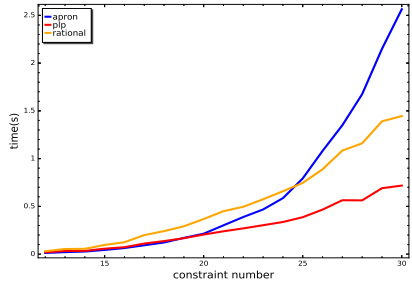}
	\caption{CN=[12,30],VN=8,D=37.5\%,PR=62.5\%}
	\label{fig:exp_c}
\end{figure}

\paragraph{Number of variables}
Here the range of variables is 3 to 15.
\fig{\ref{fig:exp_v_1}} shows that the performance are similar for Apron and PLP when variables are fewer than 11, but after that the execution time of Apron explodes as the variable number increases. The zoomed figure is shown in \fig{\ref{fig:exp_v_2}}.

Our understanding is that the execution time of Apron is dominated by the conversion to the generator description, which is exponential in the number of constraints for polyhedra resembling hypercubes---likely for a nonempty polyhedron built from $m$ random constraints in a space of dimension less than~$m$.

\begin{figure}[!htb]
	\centering
	\begin{subfigure}{.45\textwidth}
		\centering
		\includegraphics[width=\textwidth,height=3.8cm]{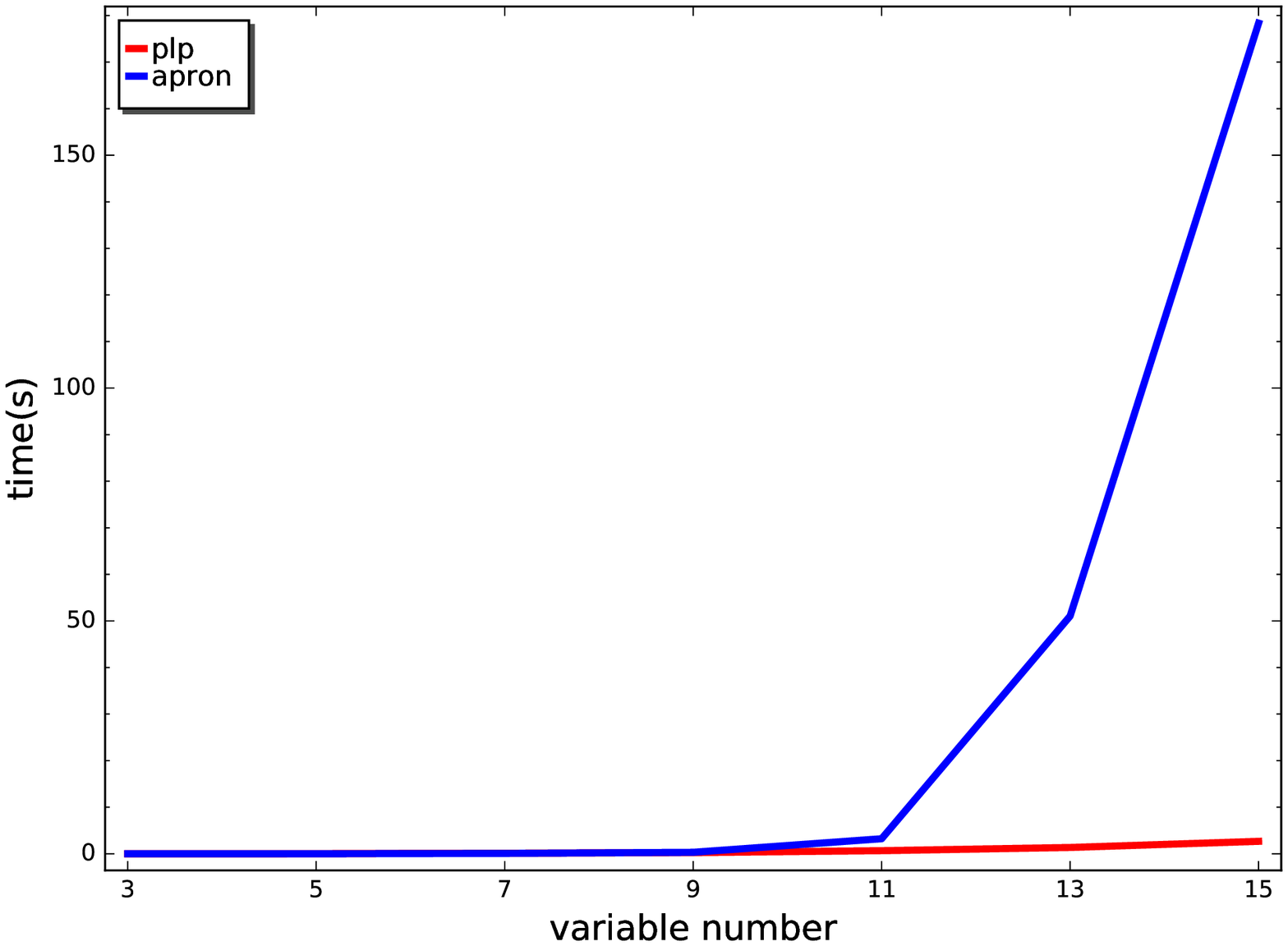}
		\caption{}
		\label{fig:exp_v_1}
	\end{subfigure}
	\begin{subfigure}{.45\textwidth}
		\centering
		\includegraphics[width=\textwidth,height=3.8cm]{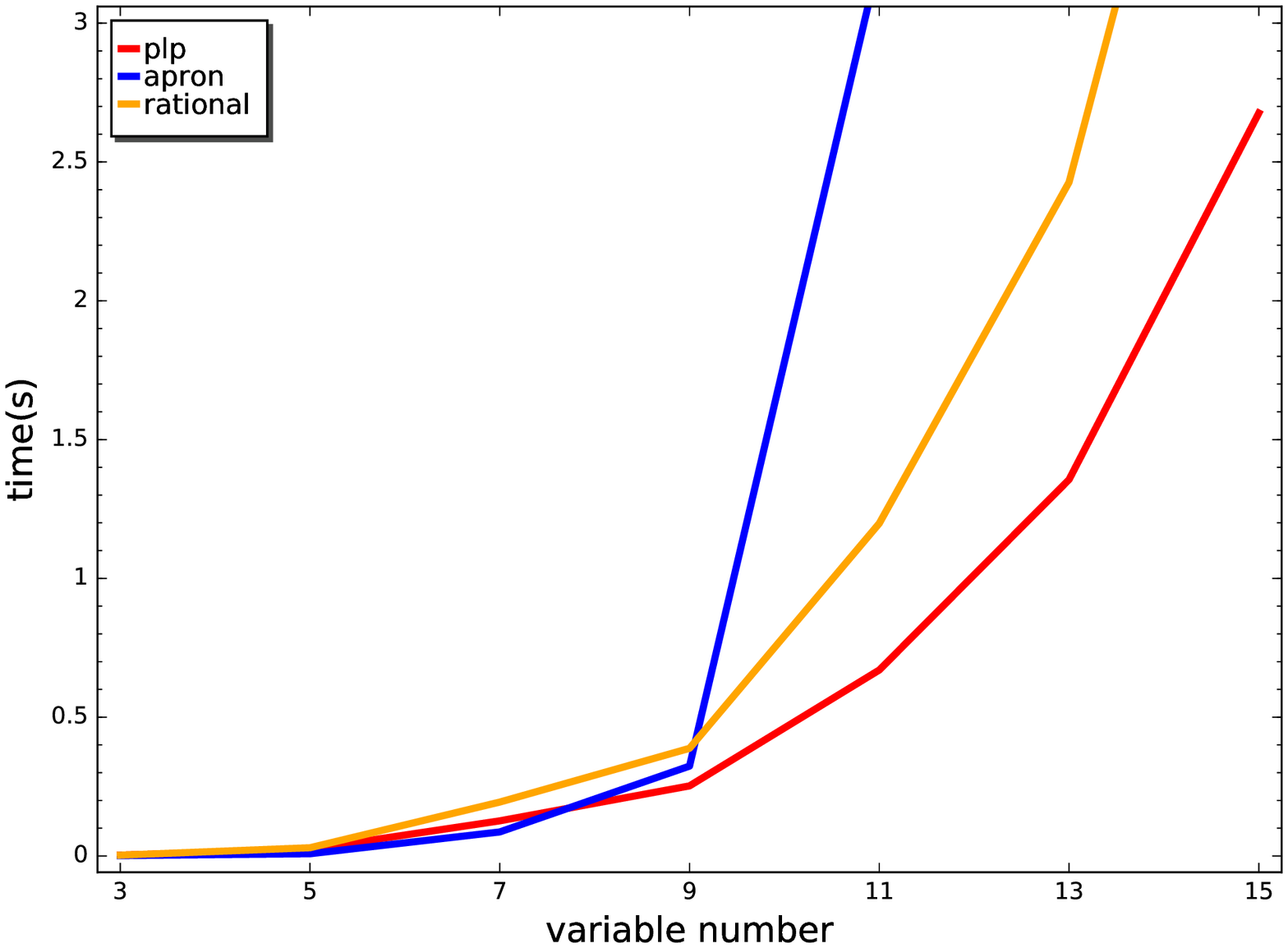}
		\caption{}
		\label{fig:exp_v_2}
	\end{subfigure}
	\caption{CN=19,VN=[3,15],D=37.5\%,PR=62.5\%}
	\label{fig:exp_v}
\end{figure}

\paragraph{Density}
The \fig{\ref{fig:exp_d}} shows the effect of density. The execution time varies for both Apron and PLP with the increase of density, with the same trend.

\begin{figure}[!htb]
	\centering
	\includegraphics[width=0.45\textwidth,height=3.8cm]{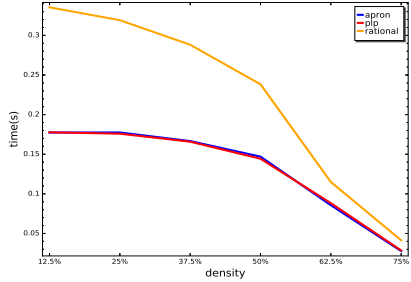}
	\caption{CN=19,VN=8,D=[12.5\%,75\%],PR=62.5\%}
	\label{fig:exp_d}
\end{figure}

\subsection{Experiments on SV-COMP benchmarks}
In this experiment we used the analyzer Pagai \cite{henry_2012_pagai} and SV-COMP benchmarks \cite{beyer_2019_automatic}.
We randomly selected C programs from the category of Concurrency Safety, Software System and Reach Safety.
The result is compared with NewPolka and ELINA.
In Table \ref{tableau:exp}, we show the name of programs, the number of polyhedra to be projected (Num), the total and average time (AveT) spent on projection, the average constraint number (ACN) and the average variable number (AVN).
The time is in milliseconds.
As it is shown, our algorithm has advantage over Apron when the polyhedra contain more constraints and/or in higher dimension, e.g, polyhedra in ldv-linux-3.0-module-loop and ldv-linux-3.0-bluetooth, as we get rid of maintaining double description.
ELINA is the most efficient.

\begin{table}
\begin{tabular}{|c|c|c|c|c|c|c|c|c|c|}
	\hline
	Program & Num & Apron & AveT & ELINA & AveT & PLP & AveT & ACN & AVN \\
	\hline 
	\rowcolor{gray!20}
	pthread-complex-buffer & 405 & 116.03 & 0.29 & 71.46 & 0.18 & 128.56 & 0.32 & 3.25 & 3.06 \\
	ldv-linux-3.0-module-loop & 10745 & \textbf{6148.74} & 0.57 & 2346.16 & 0.22 & \textbf{3969.44} & 0.37 & 3.16 & \textbf{16.19}  \\ 
	\rowcolor{gray!30}
	ssh-clnt-01.csv & 17655 & 5081.45 & 0.29 & 3123.7 & 0.18 & 5664.97 & 0.32 & 3.53 & 2.61 \\ 
	ldv-consumption-firewire & 30650 & 13763.71 & 0.45 & 8574.01 & 0.28 & 21493.57 & 0.7 & 7.19 & 6.14 \\ 
	\rowcolor{gray!20}
	busybox-1.22.0-head3 & 18340 & 13686.23 & 0.75 & 6930.74 & 0.38 & 23971.92 & 1.31 & 10.94 & 6.14 \\ 
	ldv-linux-3.0-magicmouse & 20 & 6.6 & 0.33 & 4.24 & 0.21 & 10.3 & 0.52 & 5 & 5 \\ 
	\rowcolor{gray!20}
	ldv-linux-3.0-usb-input & 1230 & 327.22 & 0.27 & 198.0 & 0.16 & 356.71 & 0.29 & 3 & 2 \\ 
	bitvector-gcd & 240 & 78.14 & 0.33 & 46.06 & 0.19 & 174.55 & 0.73 & 5 & 3 \\ 
	\rowcolor{gray!20}
	array-example-sorting & 5395 & 1769.75 & 0.33 & 1081.45 & 0.2 & 3413.21 & 0.63 & 4.78 & 3.67 \\ 
	ldv-linux-3.0-bluetooth & 15250 & \textbf{3898819.28} & 255.66 & 37477.14 & 2.46 & \textbf{190001.11} & 12.46 & \textbf{20.62} & \textbf{17.66} \\ 
	\rowcolor{gray!20}
	ssh-srvr-01 & 82500 & 35806.67 & 0.43 & 20170.35 & 0.24 & 98763.8 & 1.2 & 5.91 & 4.68 \\ 
	\hline
\end{tabular}
\caption{Performance on SV-COMP benchmarks.}
\label{tableau:exp}
\end{table}

\subsection{Analysis}
We conclude that our approach has remarkable advantage over Apron for projecting polyhedra in large dimension (large number of constraints or/and variables); it is not good choice for solving problems with few constraints in small dimension.

Our serial algorithm is less efficient than ELINA, but our approach is parallelable and is able to speed up with multiple threads.

\section{Conclusion and future work}
We have presented an algorithm to project convex polyhedra via parametric linear programming. It internally uses floating-point numbers, and then the exact result is constructed over the rationals.
Due to floating-point round-off errors, some faces may be missed by the main pass of our algorithm.
However, we can detect this situation and recover the missing faces using an exact solver.

We currently store the regions that have been explored into an unstructured array; checking whether an optimization direction is covered by an existing region is done by linear search.
This could be improved in two ways:
\begin{inparaenum}[i)]
\item regions corresponding to the same optimum (primal degeneracy) could be merged into a single region;
\item regions could be stored in a structure allowing fast search.
\end{inparaenum}
For instance, we could use a binary tree where each node is labeled with a hyperplane, and each path from the root corresponds to a conjunction of half-spaces; then each region is stored only in the paths such that the associated half-spaces intersects the region.

\printbibliography

\end{document}